\documentclass{amsart}
\usepackage{graphicx} 
\usepackage{amsfonts}
\usepackage{amssymb}
\usepackage[T1]{fontenc}
\usepackage{amsmath}
\usepackage{amsthm}
\usepackage{dsfont}
\usepackage{color}
\usepackage{caption}
\usepackage{enumitem}

\newtheorem{theorem}{Theorem}[section]

\newtheorem{proposition}[theorem]{Proposition}
\newtheorem{corollary}[theorem]{Corollary}

\newtheorem{question}[theorem]{Question}
\newtheorem{df}{Definition}

\newtheorem{problem}[df]{Problem}
\theoremstyle{definition}
\newtheorem{example}[df]{Example}

\newcommand{\N}{\mathbb N}

\newcommand{\R}{\mathbb R}

\newcommand{\ve}{\varepsilon}

\newcommand{\inte}{\mathrm{int}\,}

\title{Achievement sets of series in $\R^2$}
\author{Mateusz Kula, Piotr Nowakowski}

\address{Faculty of Science and Technology\\
University of Silesia in Katowice\\
Bankowa 14 \\
40-007 Katowice\\
Poland\\
ORCID: 0000-0002-2811-1744}
\email{mateusz.kula@us.edu.pl}

\address{Faculty of Mathematics and Computer Science
\\University of \L\'{o}d\'{z}
\\Banacha 22,
90-238 \L\'{o}d\'{z}
\\Poland
\\ORCID: 0000-0002-3655-4991}
\email{piotr.nowakowski@wmii.uni.lodz.pl}
\makeatletter
\@namedef{subjclassname@2020}{%
  \textup{2020} Mathematics Subject Classification}
\makeatother

\subjclass[2020]{11B99, 40A30, 51M15, 54F65} 
\keywords{achievement sets, Cantor sets, Cantorvals, center of distances, spectre of a set}

\begin{document}
\begin{abstract}
We examine properties of achievement sets of series in $\mathbb{R}^2$. We show several examples of unusual sets of subsums on the plane. We prove that we can obtain any set of $P$-sums as a cut of an achievement set in $\mathbb{R}^2.$ We introduce a notion of the spectre of a set in an Abelian group, which generalizes the notion of the center of distances. We examine properties of the spectre and we use it, for example, to show that the Sierpiński carpet is not an achievement set of any series. 
\end{abstract}
\maketitle

\section{Introduction}

In a Banach space, given a sequence $x = (x_n)$ such that the series $\sum_n x_n$ is absolutely
convergent, the following set (of all subsums) is well defined:
$$
E(x):= \left\{\sum_{n=1}^\infty \ve_n x_n \colon (\ve_n) \in\{0,1\}^\N\right\}.
$$
It is called the \textit{achievement set} of the series $\sum_n x_n$ (or the sequence $(x_n)$).
 The investigations of achievement sets were initiated over a century ago by Kakeya in \cite{Ka}. His main goal was to find all possible topological types of achievement sets in $\R$. He conjectured that such sets may have three possible forms: a finite set, a homeomorphic copy of the Cantor set or a finite union of closed intervals. This occurred to be false. In {\cite{GN}} Guthrie and Nymann proved that achievement sets of absolutely convergent series on the real line have one of four possible forms (the proof was later corrected in \cite{NS0}). Apart from the types mentioned by Kakeya, the achievement set can be a set that was later named an M-Cantorval (and now is often called just a Cantorval). 

 The name M-Cantorval first appeared in  the paper \cite{MO} devoted to arithmetic sums of Cantor sets on the real line. We formulate their definition in equivalent form: a non-empty compact set $M\subset \R$ is called an \textit{M-Cantorval}, whenever it is regularly closed and both endpoints of any maximal interval $I\subset M$ are accumulation points of $M\setminus I$, compare \cite[Theorem 21.17]{BFPW1}.
 The classical example of an M-Cantorval is the set that we obtain, when we add to the classical ternary Cantor set all open intervals that are removed in odd steps of the construction of this set. All M-Cantorvals are homeomorphic, see \cite{BFPW1}.

Any nonempty, totally disconnected, perfect and compact set $C \subset \R^n$, where $n \in \N$, is homeomorphic to the classical ternary Cantor set (see \cite[Corollary 30.4]{Wil})
and we call such a set simply a Cantor set.

Achievement sets are strongly connected with research concerning algebraic sums or differences of Cantor sets (e.g. \cite{AI}, \cite{FF}, \cite{FN}), which are important also in other areas, for example in the studies of homoclinic bifurcations of dynamical systems (see \cite{PT}). The set of subsums of a series may be also seen as a range of some purely atomic measure (see \cite{recover}, \cite{F}). Many other interesting papers devoted to achievement sets on the line were published in recent years (e.g. \cite{BBFS}, \cite{BFPW1}, \cite{J}).
There are still many interesting open questions in this area. We, however, focus on achievement sets on the plane. There are a few papers on this topic (see \cite{BG}, \cite{M}, \cite{M2}, \cite{BGM}). Nevertheless, research on achievement sets of series in $\R^2$ is significantly less developed than in $\R$. The main reason is that there are a lot of difficulties, which do not occur on the line, but they appear on the plane.
One of the natural problems is to find a full classification of all topological types of sets of subsums in $\R^2$ analogous to that of Guthrie and Nymann on the line. In this paper we will show why it is far more difficult. 



The paper is organized as follows. In the next section we make some easy observations and we give some general facts concerning achievement sets on the plane. In the third section we introduce the notion of topological type for some achievement sets and give examples of achievement sets which are of the same type, but have different topological structure. In the fourth section we prove the main theorem of the paper which helped us find examples of new topological types of sets of subsums in $\R^2$. The final section is devoted to a new notion of the spectre of a set, which can be used to examine achievement sets.  

We finish this section by introducing some notation, which is used throughout the paper.

Let $C \subset \R$. Any 
component of 
the set $\R \setminus C$ is called a \textit{gap} of $C$. A component of $C$ is called \textit{proper}, if it is not a one-point set.

Let us recall the definitions of other types of Cantorvals (see \cite{MO}).
A perfect set $A \subset \R$ is called an L-Cantorval (respectively an R-Cantorval), if it has infinitely many gaps, the left (right) endpoint of any gap is accumulated by gaps and proper components of $A$, and the right (left) endpoint of any gap is an endpoint of a proper component of $A$.

For any finite set $A$ by $|A|$ we denote the cardinality of $A$.

Every sequence
$$(s_1,s_2,\dots,s_k;q):=(s_1q,s_2q,\dots, s_kq,s_1q^2,s_2q^2, \dots, s_kq^2,\dots),$$
where $ k \in \N$, $s_1, s_2, \dots, s_k\in \R$, $q \in (0,1)$,
is called a multigeometric sequence.

For $A \subset \R^2$, $a\in \R$ we set $pr_x(A): = \{x \in \R\colon \exists_{y\in \R}\,\, (x,y) \in A\},$
$pr_y(A): = \{y \in \R\colon \exists_{x\in \R}\,\, (x,y) \in A\},$
$A_a := \{x \in \R\colon (x,a) \in A\},$ $A^a := \{y \in \R\colon (a,y) \in A\}.$

We say that a point $a \in E(x)$ has $k$ \textit{representations}, if there exist exactly $k$ different sets $A \subset \N$ such that $a = \sum_{i\in A} x_i.$

\section{Simple properties}
Let us start with some simple observations regarding the achievement sets on the plane. They were also made in \cite{BG}.
\begin{proposition} \label{Cart}
    Let $x=(x_n) \in \R ^{\N}$, $y=(y_n)\in  \R^{\N}$ be sequences such that the series $\sum_n x_n$, $\sum_n y_n$ are absolutely convergent. Then 
    \begin{itemize}
        \item[(i)] $E(x,y)\subset E(x) \times E(y)$;

\item[(ii)] $pr_xE(x,y) = E(x)$, $pr_yE(x,y) = E(y)$;

\item[(iii)] $E(x,y)$ is compact;

\item[(iv)] $E(x,y)$ is symmetric with respect to $(\frac{1}{2}\sum_{n\in \N} x_n, \frac{1}{2}\sum_{n\in \N} y_n)$.   
    \end{itemize}
\end{proposition}
Reapeating the reasoning from \cite{MM}, we also get the following property.
\begin{proposition}
    Let $x=(x_n) \in \R ^{\N}$, $y=(y_n)\in  \R^{\N}$ be sequences such that the series $\sum_n x_n$, $\sum_n y_n$ are absolutely convergent. If every point of $E(x,y)$ has a unique representation, then $E(x,y)$ is a Cantor set.  
\end{proposition}

\begin{corollary}
    Let $x=(x_n) \in \R ^{\N}$, $y=(y_n)\in  \R^{\N}$ be sequences such that the series $\sum_n x_n$, $\sum_n y_n$ are absolutely convergent. If every point in $E(x)$ or every point in $E(y)$ has a unique representation, then $E(x,y)$ is a Cantor set.
\end{corollary}
    \begin{example} \label{12Cantor}
Consider the sequence $(x_n,y_n)$, where $x_n = \frac{1}{2^n}$ for all $n \in \N$ and $y_n$ is such that for all $n \in \N$, $y_n \neq \sum_{i>n} y_i $. We will show that $E(x_n, y_n)$ is a Cantor set. We will do it, by showing that each point of $E(x_n,y_n)$ has a unique representation.
Consider the set $E(x_n)=[0,1]$. It is known that each point of $[0,1]$ has a unique infinite representation and only points that can be written as $x = \sum_{n \in A} \frac{1}{2^n}$ for some finite set $A \subset \N$, have exactly two representations. Specifically, if $x=\sum_{i=1}^k \frac{1}{2^{n_i}}$ for some finite increasing sequence of natural numbers $(n_i)$, then also $x=\sum_{i=1}^{k-1}\frac{1}{2^{n_i}} + \sum_{i=n_{k}+1}^{\infty} \frac{1}{2^i}$. 

Fix $(x,y) \in E(x_n,y_n)$ and suppose that it has two representations $A,A'\subset \N$. If the point $x$ has a unique representation, then $A=A'$. Hence if $A\neq A'$, then $A=\{n_1,\ldots, n_k\}$ is a nonempty finite set (the sequence  $(n_i)$ is assumed to be increasing) and $A'=\{n_1,\ldots, n_{k-1}\}\cup\{n_k+1,n_k+2,\ldots\}$. Then we have $y = \sum_{i=1}^k y_{n_i}$ and $y=\sum_{i=1}^{k-1} y_{n_i}+ \sum_{i=n_k+1}^\infty  y_i$, but this implies $y_{n_k}=\sum_{i=n_k+1}^\infty  y_i$, which contradicts our assumption. Thus, $E(x_n,y_n)$ is a Cantor set.


\end{example}
We would like to formulate the following problem.
\begin{problem}
    Characterize $E(p^n,q^n)$ for all $p, q \in (0,1)$. 
\end{problem}
From Example \ref{12Cantor} we can deduce that if $q\in (\frac{1}{2},1)$, then the set $E(\frac{1}{2^n},q^n)$ is a Cantor set. Also, if $p,q \in (0,1)$ and $p$ or $q$ is less than $\frac{1}{2}$, then $E(p^n,q^n)$ is a Cantor set. If $p=q \geq \frac{1}{2}$, then $E(p^n,p^n)$ is a segment. The question is what type of set is $E(p^n,q^n)$, if $p>q>\frac12.$ We conjecture that it is a Cantor set for every such $p$ and $q$. 

In the end of the section we would like to discuss the topic of achievement sets which are graphs of some functions.
Observe that for any absolutely convergent series $\sum_n x_n$ the set $E(x_n,x_n)$ is a graph of some function.
We also have the following property.
\begin{proposition}
  Let $x=(x_n) \in \R ^{\N}$, $y=(y_n)\in  \R^{\N}$ be sequences such that the series $\sum_n x_n$, $\sum_n y_n$ are absolutely convergent. If each point in $E(x)$ has a unique representation, then $E(x,y)$ is a graph of some function. Moreover, if also every point of $E(y)$ has a unique representation, then this function is a bijection from $E(x)$ onto $E(y)$.
\end{proposition}
Let $(x_n)$ be a nonincreasing sequence of positive numbers such that the series $\sum_n x_n$ is absolutely convergent series in $\R$. We say that a series $\sum_n x_n$ is \textit{fast convergent}, if $a_n>r_n:= \sum_{i>n} x_i$ for all $n\in\N$.  We say
that a series $\sum x_n$ is \textit{slowly convergent}, if for any $n \in \N$ we have
$x_n \leq r_n$. It is known that if $\sum_n x_n$ is fast convergent, then $E(x_n)$ is a Cantor set and $E(x_n)$ is an interval if and only if $\sum_n x_n$ is slowly convergent. 
We were interested in the question whether there is a function $f\colon [0,1] \to [0,\infty)$ such that its graph is an achievement set of some series in $\R^2$ and $f$ is not linear. We found the answer in \cite{N}. 
\begin{theorem} \cite{N}
 Let $(x_n)$, $(y_n)$ be such that the series $\sum_n x_n$, $\sum_n y_n$ are absolutely convergent and slowly convergent. If $E(x_n,y_n)$ is a graph of some function $f$, then $f(t) = a t$ for some $a>0.$ In particular, $y_n = ax_n$ for all $n\in \N$.
\end{theorem}

\section{Examples}

We are going to discuss examples of sets that are achievement sets on the plane. Our main interest is in different topological types.

Clearly, given a one-dimensional achievement set $E(x_n)$, it is possible to get its isometric copy on the $x$-axis on the plane by taking the sequence $((x_n,0))_n$.
Also, given a linear function $F\colon \R^2\rightarrow\R^2$ and an achievement set $E=E(x_n,y_n) \subset \R^2$, it is possible to obtain the image $F(E)$. Indeed, it suffices to take the sequence $(F(x_n,y_n))_n$.

Like in any Banach space, given a sequence of achievement sets $(E_1,E_2,\ldots)$ such that  $\sum_{n=1}^\infty \sum_{k=1}^{\infty}\|x^n_k\|<\infty$, where $E_n=E((x^n_k)_{k})$, one can obtain a set $\bigcup_{n\in\N}(E_1+\dots+E_n)$ as an achievement set. Indeed, it is enough to take the diagonal enumeration
$(x^1_1, x^2_1, x^1_2, x^3_1, x^2_2, x^1_3,\ldots)$. In particular, we can obtain finite algebraic sums of any achievement sets.

From Guthrie-Nymann classification theorem (\cite[Theorem 1]{GN}, \cite{NS0}) we know that the achievement set of an absolutely convergent series on $\R$ is always of one of the following four topological types: a finite set, a finite union of closed intervals, a Cantor set or an M-Cantorval.
On the plane, it is possible to get the Cartesian product of any two one-dimensional achievement sets (see also \cite{BG}). If $\sum_n x_n$, $\sum_n y_n$ are absolutely convergent series on $\R$, then we have
$$E((x_1,0),(0,y_1),(x_2,0),(0,y_2),\ldots) = E(x_n) \times E(y_n).$$ As the consequence, we get the following proposition.
\begin{proposition} \label{types}
Let $I :=[0,1]$, $C$ be a Cantor set, and $M$ be an M-Cantorval.
    For any of the following sets:
\begin{itemize}
    \item a one-point set;

    \item $I$;

    \item $C$;

    \item $M$;

    \item $I\times I$;

    \item $I \times C$;

    \item $I \times M$;

    \item $C \times M$;

    \item $M \times M,$
\end{itemize}
there exists an absolutely convergent series $\sum_n x_n$ is $\R^2$ such that $E(x_n)$ is homeomorphic to that set. 
\end{proposition}
 Some natural examples of achievement sets turn out to be finite unions of sets of the same type. Whenever we write that some set is of type $W \times V$, where $W,V \in \{I,C,M\},$ we mean that it is a finite union of sets homeomorphic to the set $W \times V$.

 All of the mentioned types differ in terms of topological properties and a natural question arises.
 \begin{question}\label{quest}
 Is any achievement set on the plane of one of the types from Proposition \ref{types}? 
 \end{question}
 In \cite{BG} the authors conjectured that the answer is positive for some particular family of achievement sets on the plane. In the next section we will provide an answer to the Question \ref{quest}. But first, we would like to exhibit some examples of achievement sets which are of types from Proposition \ref{types}, but at the same time have connected components that are neither homeomorphic to one-point set, a closed interval nor a closed square. We call such connected sets \textit{unusual}. In this section, $M$ denotes the specific M-Cantorval, namely the achievement set of the sequence $(x_n)=(\frac34, \frac12,\ldots, \frac 3{4^n}, \frac 2{4^n},\ldots)$. It is sometimes called Guthrie--Nymann Cantorval, see \cite{GN} or \cite{BPW}.
\begin{enumerate}[label=\arabic*.]
	\item \textbf{Finite unions of unit squares.} The unit square $[0,1]^2$ is an achievement set of a sequence $s=\left((\frac12,0),(0,\frac12),(\frac14,0),(0,\frac14),\ldots\right)$. Consider the following finite sequences:
	\begin{enumerate}[label=(\alph*)]
		\item $(\underbrace{(1,1),\ldots,(1,1)}_{n\text{ times}})$ for some $n\in\N$.
  
		\item $(\underbrace{(1,1),\ldots,(1,1)}_{n\text{ times}},\underbrace{(-1,1),\ldots,(-1,1)}_{m\text{ times}})$ for some $n,m\in \N$.
  
		\item $(\underbrace{(\frac23,\frac23),\ldots,(\frac23,\frac23)}_{n\text{ times}},\underbrace{(-1,1),\ldots,(-1,1)}_{m\text{ times}})$ for some $n,m\in \N$.
  
		\item $((-\frac23,\frac23),\underbrace{(\frac23,\frac23),\ldots,(\frac23,\frac23)}_{n\text{ times}})$ for some $n\in\N$.
  
	\end{enumerate}
 \begin{figure}
\begin{minipage}{.5\textwidth}
  \centering
    \includegraphics[scale=0.35]{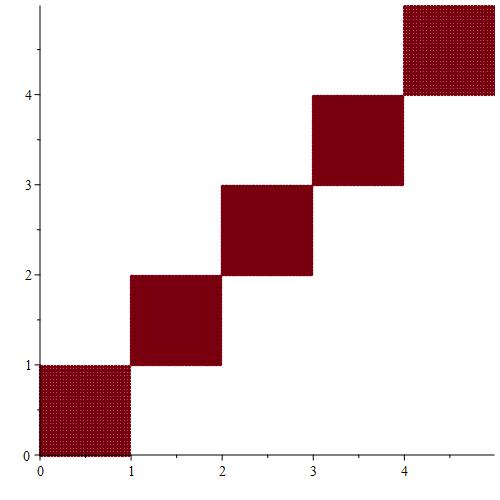}
    \caption{Achievement set from 1.(a) with $n=4$.}
    \label{fig:square-a}
\end{minipage}%
\begin{minipage}{.5\textwidth}
  \centering
  \includegraphics[scale=0.35]{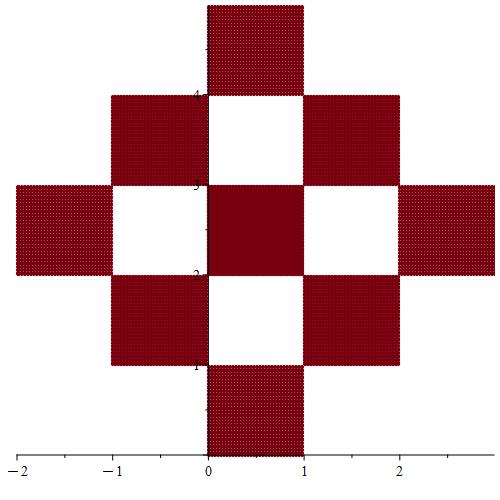}
    \caption{Achievement set from 1.(b) with $n=2$ and $m=2$.}
    \label{fig:square-b}
\end{minipage}
\begin{minipage}{.5\textwidth}
  \centering
   \includegraphics[scale=0.35]{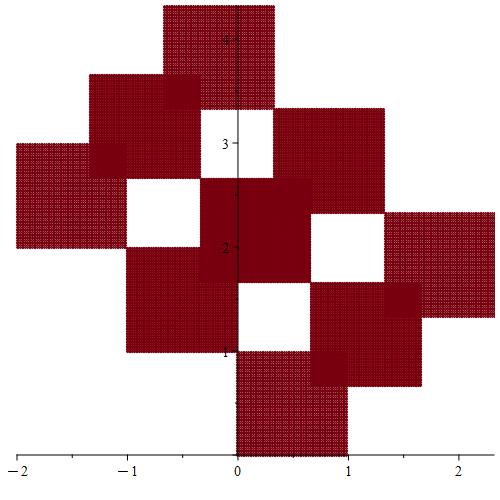}
    \caption{Achievement set from 1.(c) with $n=2$ and $m=2$.}
    \label{fig:square-c}
\end{minipage}%
\begin{minipage}{.5\textwidth}
  \centering
   \includegraphics[scale=0.35]{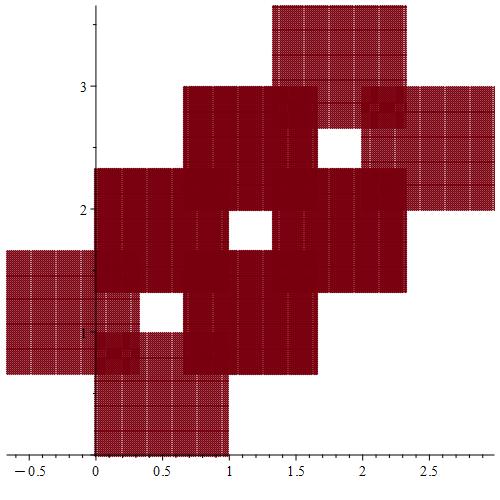}
    \caption{Achievement set from 1.(d) with $n=3$.}
    \label{fig:square-d}
\end{minipage}%
\end{figure}%
 
	The achievement set of a sequence obtained by prepending $s$ with any of the sequences as in (a)--(d), is an unusual connected set (see Figures \ref{fig:square-a}--\ref{fig:square-d}).
	Moreover, no two different sets obtained in this way are homeomorphic.
	\item \textbf{Infinite unions of unit squares.} 
	Since the algebraic sum of finitely many achievement sets is an achievement set, by taking the sequence
  $$(\underbrace{(-1,1),\ldots,(-1,1)}_{m\text{ times}},(\frac{1}{2},0),(0,\frac{1}{2}),(x_1,x_1),(\frac{1}{4},0),(0,\frac{1}{4}),(x_2,x_2),\ldots),$$ 
 we can obtain the algebraic sum $E$ of the unit square $[0,1]^2$, the finite set $\{(-k,k)\colon k\in\{0,\ldots, m\}\}$ and the set $\{(x,x)\colon x\in E(x_n)\}$,
where $(x_n)$ is a sequence of real numbers such that $\sum x_n$ is absolutely convergent. 
	\begin{enumerate}[label=(\alph*)]
		\item If $(x_n)=(\frac23,\ldots,\frac2{3^n},\ldots)$, then the set $E$ is an unusual connected set, which has infinitely many `holes', because $E(x_n)$ is a Cantor set (see Figure \ref{fig:inf-square-a}).

		\item If $(x_n)=(\frac34, \frac12,\ldots, \frac 3{4^n}, \frac 2{4^n},\ldots)$, then $E$ is another example of an unusual connected set, because $E(x_n)$ is an M-Cantorval (see Figure \ref{fig:inf-square-b}).
	\end{enumerate}
Both in (a) and (b) the achievement set E is of type $I\times I$, that is, it is a finite union of sets  homeomorphic to the closed square.
\begin{figure}[t]
 \centering
\begin{minipage}{.5\textwidth}
  \centering
    \includegraphics[scale=0.33]{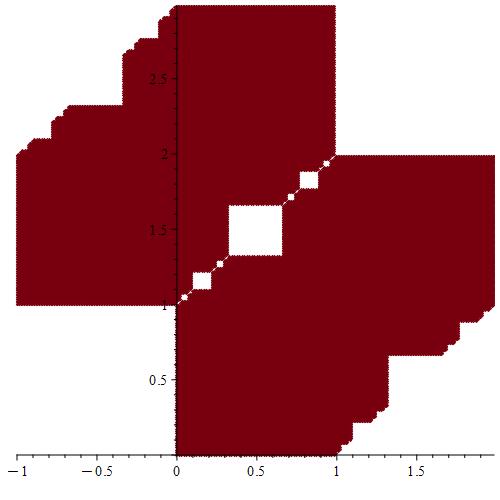}
    \caption{Achievement set from 2.(a) with $m=1$.}
    \label{fig:inf-square-a}
\end{minipage}%
\begin{minipage}{.5\textwidth}
  \centering
  \includegraphics[scale=0.33]{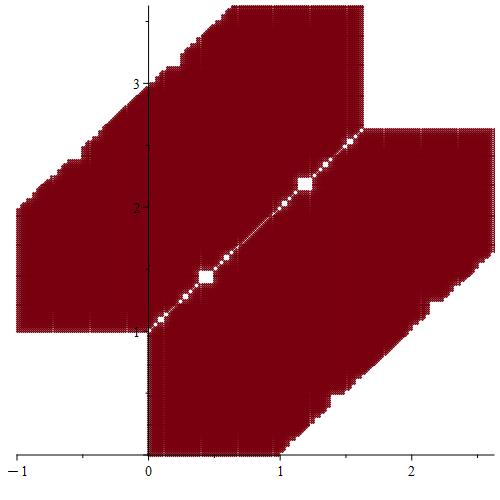}
    \caption{Achievement set from 2.(b) with $m=1$.}
    \label{fig:inf-square-b}
\end{minipage}
\end{figure}%
 \begin{figure}[b]
\begin{minipage}{.5\textwidth}
 \centering
      \includegraphics[scale=0.33]{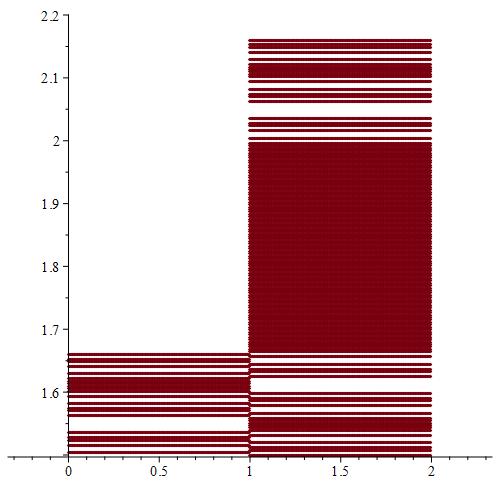}
      \caption{Achievement set from 3.(a).}
      \label{fig:ixm-a}
\end{minipage}%
\begin{minipage}{.5\textwidth}
  \centering
      \includegraphics[scale=0.33]{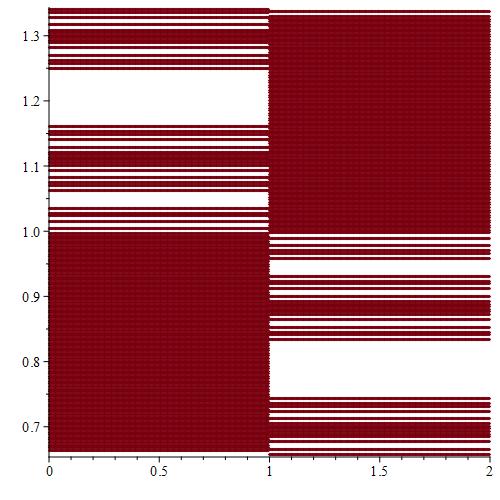}
      \caption{Achievement set from 3.(b).}
      \label{fig:ixm-b}
\end{minipage}
\begin{minipage}{.5\textwidth}
  \centering
      \includegraphics[scale=0.33]{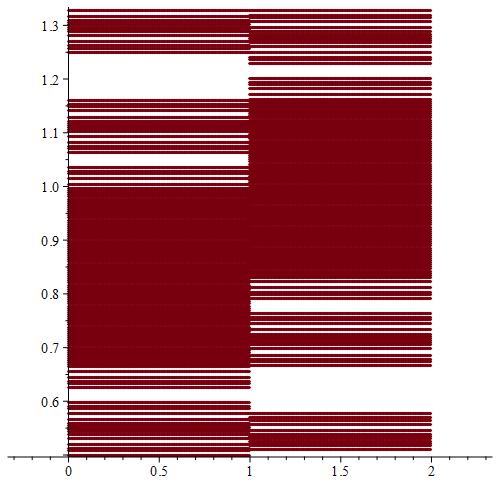}
      \caption{Achievement set from 3.(c).}
      \label{fig:ixm-c}
\end{minipage}%
\end{figure}%
	\item \textbf{Finite unions of $I\times M$.} The product $I\times M$ as well as the set $E=I\times M\cup((x,y)+I\times M)$ are achievement sets for any $(x,y)\in\R^2$. 
	\begin{enumerate}[label=(\alph*)]
		\item If $(x,y)=(1,1)$, then the connected component $C$ of a point $(1,\frac53)$ in $E$ is unusual, because it is the union of an arc and a closed rectangle with one point in common (see Figure \ref{fig:ixm-a}).
 
		\item If $(x,y)=(1,\frac 13)$, then the connected component $C$ of a point $(1,1)$ in $E$ is unusual (see Figure \ref{fig:ixm-b}). 

		\item If $(x,y)=(1,\frac 16)$, then the connected component $C$ of a point $(1,1)$ in $E$ is unusual (see Figure \ref{fig:ixm-c}). 

	\end{enumerate}

 Moreover, no two different sets  referred to in (a)--(c) are homeomorphic. To see this, note that if $h\colon E\rightarrow E'$ is a homeomorphism, where $E$ and $E'$ are subsets of $\R^n$, then $h[\inte E]=\inte E'$ (compare \cite[Corollary 5.18]{Ful}).
   In the set $C$ in (a) there is exactly one point $x$ such that $x\in\overline{\inte C}$ and $C\setminus\{x\}$ is disconnected. In the sets $C$ in (b) and (c) there are infinitely many such points.
  The interior of $C$ in (b) has two connected components, while the interior of $C$ in (c) has only one component.
\end{enumerate}

\section{Characterization attempts}
In this section we will answer the Question \ref{quest}. But first, let us recall the notion of the set of $P$-sums (\cite{NS}).

Let $P \subset \R$ be a finite set and let $a=(a_n)$ be a sequence of real numbers such that the series $\sum_n a_n$ is absolutely convergent, and $|a_n| \geq |a_{n+1}| > 0$ for all $n \in \N$. Let
$$S(P,a) := \left\{\sum_{n=1}^\infty \ve_n a_n\colon (\ve_n) \in P^\N \right\}.$$
The set $S(P;a)$ is called the \textit{set of} $P$-\textit{sums} for the sequence $a$.
In particular, if $P = \{0,1\},$ then $S(P,a) =E(a).$

To answer Question \ref{quest} we were looking for some topological characterizations of types of achievement sets from Proposition \ref{types}. One of the ideas was to use cuts of sets. For example, we conjectured that the achievement set is of type $M \times M$ if and only if every cut of the set is either finite or is an M-Cantorval. Although we didn't manage to prove it, the reasoning led us to the following theorem. 

\begin{theorem} \label{P-sums}
    Let $P$ be a finite set of real numbers such that $0 \in P$ and let $a=(a_n)$ be a sequence of real numbers such that the series $\sum_n a_n$ is absolutely convergent and $|a_n| \geq |a_{n+1}| > 0$ for all $n \in \N$. There exists a sequence $(x_n,y_n) \in (\R^{2})^\N$ such that $E(x_n,y_n)_0 = S(P,a).$
\end{theorem}
\begin{proof}
If $P = \{0\}$, then it suffices to take $(x_n,y_n) = (0,0).$ Suppose that $|P| = k$ for some $k > 1$. Let $P_0 = P \setminus \{0\}.$
    Enumerate elements of the set  $P_0$ as follows $P_0 = \{p_1, p_2, \dots, p_{k-1}\}$, where $p_i < p_{i+1} $ for $i < k-1.$ Let $b_n = \frac{1}{(k+2)^n}$ for $n \in \N\cup\{0\}$.
    For $n \in \N \cup \{0\}$ put
    $x_{kn+1}=0$, $y_{kn+1}=b_n$,
    $x_{kn+i}=p_{i-1}a_{n+1},$
    $y_{kn+i}=-b_n$ for $i \in \{2,3,\dots, k\}.$
    We will show that 
    $E(x_n,y_n)_0 = S(P,a).$
    
    Let $w \in E(x_n,y_n)_0.$ This means that $(w,0) \in E(x_n,y_n).$ Hence there is $A \subset  \N$ such that $\sum_{i\in A}x_i=w$ and $\sum_{i \in A} y_i = 0.$ If $A = \emptyset$, then $w=0 \in S(P,a).$

    Assume that $A \neq \emptyset$. 
    We will show that for any $n \in \N \cup \{0\}$ 
    \begin{equation}\label{kn+1}
        kn+1 \in A \Leftrightarrow \exists!_{i\in \{2,3,\dots,k\}} kn+i \in A \Leftrightarrow \exists_{i\in \{2,3,\dots,k\}} kn+i \in A.
    \end{equation}

    We proceed by induction. Fix $m\in\N\cup\{0\}$ and assume that the above condition holds for all $n<m$. First observe that, since $y_{kn+1}=b_n$ and $y_{kn+i} = -b_n$ for $i\in\{2,3,\dots,k\}$, the inductive hypothesis implies $\sum_{i\in A,i<km+1} y_i = 0$.

    Assume that $km+1 \in A$ and suppose that $km+i \notin A$ for $i\in\{2,3,\dots,k\}.$  Then 
    $$0=\sum_{i \in A, i>km} y_i \geq y_{km+1}+ \sum_{n=m+1}^\infty -(k-1)b_n = b_m-(k-1)\sum_{n=m+1}^\infty \frac{1}{(k+2)^n}=$$$$=\frac1{(k+2)^m}-\frac{k-1}{(k+2)^{m+1}(1-\frac{1}{k+2})} =\frac2{(k+1)(k+2)^m}>0;$$
    a contradiction. So, there is $i\in\{2,3,\dots,k\}$ such that $km+i \in A$.
    Now, suppose that either $km+i, km+j \in A$ for $i,j \in \{2,3,\dots,k\},$ $i\neq j$ or $km+1 \notin A$ and $km+i \in A$ for some $i\in \{2,3,\dots,k\}.$ Then $$\sum_{i\in A, km<i\leq km+k}y_i\leq -b_m,$$ and consequently
    $$0=\sum_{i \in A, i>km} y_i \leq -b_m + \sum_{n>m} y_{kn+1} = -\frac{1}{(k+2)^m}+ \sum_{n=m+1}^\infty \frac{1}{(k+2)^n} = \frac{1}{(k+2)^m}\left(\frac1{k+1} -1\right) < 0;$$
    a contradiction, which finishes the proof of (\ref{kn+1}) for $m$. By the induction principle, we get that (\ref{kn+1}) is satisfied for all $n \in \N \cup \{0\}.$
    
We have $w=\sum_{n\in A} x_n$. By (\ref{kn+1}), for any $n \in \N\cup\{0\}$ there is at most one $i \in \{2,3, \dots, k\}$ such that $nk+i \in A$. We also have $x_{nk+i} = p_{i-1}a_{n+1}.$  Thus, since $0 \in P$ and $x_{nk+1} = 0$, we get that for any $n \in \N\cup\{0\}$ there is $q_{n+1} \in P$ such that 
$$\sum_{i\in A, kn < i \leq kn+k}x_i=q_{n+1}a_{n+1}.$$
Thus, $w=\sum_{n=1}^\infty q_na_n \in S(P,a).$
    
Now, let $w \in S(P,a)$. Then $w=\sum_{n=1}^\infty q_na_n$, for some sequence $(q_n) \in P^\N.$
    Let $$B := \{n \in \N \colon q_n \neq 0\},$$
    and for $n\in B$ let $i_n \in \{2,3, \dots,k\}$ be such that $q_n=p_{i_n-1}.$
    Let
    $$A_1:= \{k(n-1)+1\colon n \in B\}, $$
    $$A_2:=\{k(n-1)+i_n\colon n\in B\}, $$
    $$A =A_1 \cup A_2 .$$
    
    Then $$\sum_{n \in A_1}x_n = \sum_{n\in B} x_{(n-1)k+1}= 0,$$ and so
    $$\sum_{n\in A} x_n = \sum_{n \in A_2}x_n = \sum_{n\in B} x_{(n-1)k+i_n} = \sum_{n \in B} q_n a_n = \sum_{n\in \N} q_n a_n = w.$$
    Moreover,
    $$\sum_{n\in A} y_n = \sum_{n \in A_1}y_n + \sum_{n\in A_2} y_n= \sum_{n\in B} y_{(n-1)k+1}+\sum_{n\in B} y_{(n-1)k+i_n} = \sum_{n \in B} b_n - \sum_{n \in B} b_n = 0.$$
    Hence $(w,0) \in E(x_n,y_n)$, and so $w \in E(x_n,y_n)_0.$
\end{proof}
The following example shows that there are achievement sets of other types than in Proposition \ref{types}, and so it gives a negative answer to Question \ref{quest}.
\begin{example}
Let $x=(x_n)=(0,1,2,9;\frac{1}{3})$ and
$y=(1,-1,-1,-1;\frac{1}{6}).$ 
It is known that for $P=\{0,1,2,9\}$ and $a_n = \frac{1}{3^n}$ for $n \in \N$, the set $S(P,(a_n))$ is an L-Cantorval (see \cite{NS}, \cite{NS2} and \cite{W}). From Theorem \ref{P-sums} (and its proof), we know that $E(x,y)_0 = S(P,(a_n)).$ We will show that the set $E(x,y)$ is not homeomorphic to any known types of achievement sets on $\R^2.$ Obviously, $E(x,y)$ is not a Cantor set, a finite union of intervals or an M-Cantorval. By Proposition \ref{Cart} we have $E(x,y) \subset E(x) \times E(y)$. Observe that $E(y)$ is a Cantor set. Indeed, it is well known (see \cite{BFPW1}) that for any absolutely convergent series $\sum c_n$ the sets $E(c_n)$ and $E(|c_n|)$ are homeomorphic. Also, (see \cite[Theorem 3.5.]{BFS}) if $(s_1,\ldots, s_k;q)$ is a multigeometric sequence such that $q<\frac1{|\Sigma|}$, where $\Sigma=E(s_1,\ldots, s_k,0,0,\ldots)$, then $E(s_1,\ldots, s_k;q)$ is a Cantor set. But $(|y_n|)$ is a multigeometric sequence such that $\Sigma=\{0,1,2,3,4\}$ and $q=\frac16$.  
Therefore, $E(x,y)$ has empty interior. Thus, it cannot be a set of types $I \times I$, $I \times M$ or $M \times M$. 

For any set $S$ homeomorphic to $C \times M$ consider the closed subset $T(S)\subseteq S$ equal to the union of all trivial components of $S$ and endpoints of arcs being nontrivial components of $S$. Clearly, $T(S)$ is a Cantor set. Hence, given finitely many sets $S_1,\ldots, S_n$, each one homeomorphic to $C \times M$, the union $T:=T(S_1)\cup\ldots\cup T(S_n)$ is also a Cantor set. 
 So, $T$ is totally disconnected and closed, therefore for any non-trivial connected subset $A\subseteq S_1\cup\ldots\cup S_n$ there is a point $x\in A$ and $\varepsilon>0$ such that the open ball $B(x,\varepsilon)$ is disjoint from the union of all trivial components of $S_1\cup\ldots\cup S_n$  (because this union is contained in $T$).

From \cite{NS} we know that $[0, \frac{9}{8}]\times \{0\}$ is a component of $E(x,y)$. To prove that $E(x,y)$ is not of the type $C \times M$ (or $C \times I$, which does not contain trivial components) we will show that for any $t \in (0, \frac{9}{8})$ and any $\ve >0,$ there is a trivial component of $E(x,y)$ in the open ball $B:=B((t,0),\ve)$.
Let $t \in (0, \frac{9}{8})$ and $\ve >0$. We can find $m \in \N$ and a sequence $(c_n)\in\{0,1,2,9\}^m$ such that $u:=\sum_{i=1}^m c_ia_i \in (t-\frac{\ve}{2},t+\frac{\ve}{2}).$ Then $(u,0) = \sum_{i=1}^{4m} \ve_{i} (x_i,y_i)$, where $\ve_{i} \in \{0,1\}$ are such that $\ve_{4n+2} = 1 \Leftrightarrow c_n=1$,  $\ve_{4n+3} = 1 \Leftrightarrow c_n=2$,
$\ve_{4n+4} = 1 \Leftrightarrow c_n=9$ and $\ve_{4n+1} = 1 \Leftrightarrow c_n \neq 0$.
Let $k > m $ be such that $b:= \sum_{i=k+1}^{\infty} \frac{1}{6^i} < \frac{\ve}{2}$. 
Since the sequence $(\frac{1}{6^n})$ is fast convergent, we have that $b$ has a unique representation in $E(\frac{1}{6^n})$. 
Observe that if $\sum_{i\in I}\frac1{6^{i+1}}=\sum_{i\in A}y_i$ for some subsets $I\subset \N\cup\{0\}$ and $A\subset \N$, then for any
$i\in I$ we have $4i+1\in A$ and $ 4i+2, 4i+3, 4i+4 \notin A$ (the reasoning is similar to the proof of (\ref{kn+1})).
It follows that if 
$(c,b) = \sum_{n \in A} (x_n,y_n)$ for some $A\subset \N$, then $A$ must contain all indices $4j+1$ for $j \geq k$ and it cannot contain other indices greater than $4k+1$.
Therefore, there are only finitely many terms of the form $(c,b) \in E(x,y).$ In particular, $(u,b) = \sum_{i=1}^{4m} \ve_i (x_i,y_i)+\sum_{j\geq k} (x_{4j+1},y_{4j+1}) \in E(x,y).$  Since $E(y)$ is a Cantor set, we have that $\{(u,b)\}$ is a trivial component of $E(x,y)$. We also have $(u,b) \in B$. So, $E(x,y)$ cannot be of the types $C \times M$ or $C \times I.$  
\end{example}
\begin{figure}
\begin{minipage}{.5\textwidth}
 \includegraphics[scale=0.45]{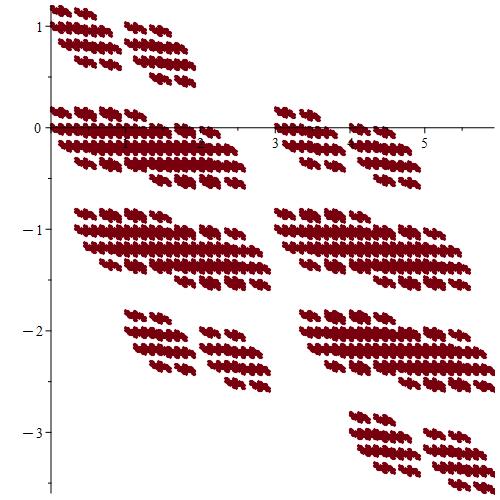}  
\end{minipage}%
\begin{minipage}{.5\textwidth}
      \includegraphics[scale=0.45]{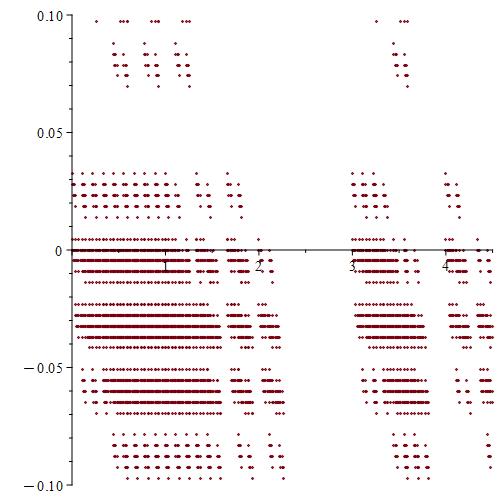}
\end{minipage}
    \caption{The pictures above show an approximation of the set $E(x,y)$ and its fragment.}
      \label{L-Cantorval-cut}
\end{figure}%

It is known that the sets of $P$-sums can have forms other than achievement set (for example, an L-Cantorval, an R-Cantorval, a infinite union of intervals with a singleton). However, the full classification of all possible sets of $P$-sums is still not known. From Theorem \ref{P-sums} we see that sets of $P$-sums play important role for achievement sets in $\R^2$. So, it is even more complicated to clasify all possible types of such sets.

\section{Spectre of a set}
In \cite{BPW} the authors introduced a notion of the center of distances. 
The center of distances of a subset $A$ of a metric space $(X,d)$ is defined as
$$C(A) = \{x\in [0,\infty)\colon \forall_{y \in A} \exists_{z\in A} d(y,z) = x\}.$$
This notion occurred to be useful for examining achievement sets in $\R$. Although the center of distances is well defined also for $\R^2$, it does not have properties that were useful in $\R$. This is why we were looking for some other similar notion which would help us in $\R^2$. Therefore, we define the spectre of a set.

Let $(X,+)$ be an Abelian group. Let $A \subset X$. We call the set
$$S(A) := \{x \in X\colon \forall_{y\in A}\, y+x \in A \mbox{ or } y-x \in A\}$$
the \textit{spectre} of a set $A$. 

Let us begin with  simple observations, which immediately follows from the definition of $S(A)$.
\begin{proposition}
Let $A \subset X$. If $x \in S(A)$, then $-x \in S(A)$. 
\end{proposition}  
\begin{proposition} \label{zero}
For any $A \subset X$, $0 \in S(A)$.
\end{proposition}
\begin{proposition}\label{prop3}
If $0 \in A$, then $S(A) \subset A \cup (-A)$.
\end{proposition}
\begin{proposition}
For any $A\subset X$ and $q \in X$, $S(q+A) = S(A).$
\end{proposition}
From the previous result, we get that $S$ is translation invariant. We will now show that in $\R^2$ the spectre of a set is not rotation invariant, but rotates with the set. We have even a more general result.
\begin{proposition}
Let $F\colon \R^2 \to \R^2$ be a linear bijective function. Then for any set $A\subset \R^2$, $S(F(A)) = F(S(A)).$
\end{proposition}
\begin{proof}
Let $F\colon \R^2 \to \R^2$ be a linear function,  $A\subset \R^2$, $(x,y) \in S(A)$ and $(a,b) \in A$. Then $(a+x,b+y) \in A$ or $(a-x,b-y) \in A$. Since $F(a,b)\pm F(x,y)=F(a\pm x,b\pm y)$, it follows that $F(a,b)+F(x,y)\in F(A)$ or $F(a,b)-F(x,y)\in F(A)$, that is, $F(x,y)\in S(F(A))$.

Thus, we have proved that $F(S(A))\subset S(F(A))$ for any linear function $F$ and any set $A\subset \R^2$. Now assume that $F$ is bijective. We get
$$F(S(A))\subset S(F(A))=F(F^{-1}(S(F(A))))\subset F(S(F^{-1}(F(A))))=F(S(A)).$$
\end{proof}
In particular, since any rotation mapping is bijective and linear, we have the following corollary. 
\begin{corollary}
  Let $r\colon \R^2 \to \R^2$ be a rotation mapping. Then for any set $A\subset \R^2$, $S(r(A)) = r(S(A)).$  
\end{corollary}
 
\begin{example}\label{e1}
Let $X= \R^2$ and $A = [0,1]^2$. Observe that $S(A) = \left(\{0\} \times [-\frac{1}{2},\frac{1}{2}]\right) \cup \left([-\frac{1}{2},\frac{1}{2}] \times \{0\}\right).$ Indeed, 
inclusion $\left(\{0\} \times [-\frac{1}{2},\frac{1}{2}]\right) \cup \left([-\frac{1}{2},\frac{1}{2}] \times \{0\}\right) \subset S(A)$ is obvious. Moreover, by Proposition \ref{prop3},
$S(A) \subset A \cup -A $.
We also have
$$ \{x \in \R^2\colon (0,1) + x \in A\vee (0,1) - x \in A\} = \left([0,1]\times [-1,0] \right)\cup \left([-1,0]\times [0,1]\right).$$
Since $(0,1) \in A$, we have $$ S(A) \subset (A \cup -A) \cap \{x \in \R^2\colon (0,1) + x \in A \vee (0,1)-x \in A\} $$$$=\left(\left([0,1]\times [0,1] \right)\cup \left([-1,0]\times [-1,0]\right)\right) \cap \left(\left([0,1]\times [-1,0] \right)\cup \left([-1,0]\times [0,1]\right)\right)= \left(\{0\} \times [-1,1]\right) \cup \left([-1,1] \times \{0\}\right).
$$ 
Also $(\frac{1}{2},\frac{1}{2}) \in A$, so $$\textstyle S(A) \subset \{x \in \R^2\colon \left(\frac{1}{2},\frac{1}{2}\right)  + x \in A\vee \left(\frac{1}{2},\frac{1}{2}\right)  - x \in A\} = \left[-\frac{1}{2},\frac{1}{2}\right] \times \left[-\frac{1}{2},\frac{1}{2}\right].$$
Therefore, 
$S(A) \subset \left(\{0\} \times [-\frac{1}{2},\frac{1}{2}]\right) \cup \left([-\frac{1}{2},\frac{1}{2}] \times \{0\}\right), $
and so
$S(A) = \left(\{0\} \times [-\frac{1}{2},\frac{1}{2}]\right) \cup \left([-\frac{1}{2},\frac{1}{2}] \times \{0\}\right). $

\end{example}
We would like to compare the notion of spectre of a set with the notion of center of distances introduced by Bielas, Plewik and Walczyńska in \cite{BPW}. 
From now on we will assume that $X$ is equipped with a translation invariant metric $d$ such that $(X,+)$ is a topological group with the topology induced by $d$. 

First, let us make two easy observations, which show the connection between a center of distances and a spectre of a set.
\begin{proposition}
Consider $\R$ with a natural metric. Then for any $A \subset \R$, we have $S(A) \cap [0,\infty) = C(A)$.
\end{proposition}
\begin{proposition}
Let $(X,d)$ be a metric space with a translation invariant metric and $A \subset X$. If $x\in S(A)$, then $d(x,0) \in C(A)$.
\end{proposition}
\begin{proof}
    Let $x \in S(A).$ Take any $a \in A.$ Then $a+x \in A$ or $a-x \in A.$ Since $d$ is translation invariant, we have
    $d(a+x,a)=d(a-x,a)=d(x,0).$ By the arbitrariness of $a$, $d(x,0) \in C(A).$
\end{proof}
The following results are analogous to results from the paper \cite{BBP} concerning center of distances. 
\begin{theorem}
If $A$ is nonempty and compact, then $S(A)$ is also compact.
\end{theorem}
\begin{proof}
For $a \in A$ let $f_a^1,f_a^2 \colon A \to X$ be functions given by the formulas
$$f_a^1(b) = b - a$$
$$f_a^2(b) = a-b.$$
Both functions are continuous, so, since $A$ is compact, we have that also $f_a^1(A), f_a^2(A)$ are compact. 
Therefore,
$\bigcap_{a \in A} \left(f_a^1(A) \cup f_a^2(A) \right)$ is also compact. To finish the proof we will show that $S(A) = \bigcap_{a \in A} \left(f_a^1(A) \cup f_a^2(A) \right)$.

Note that $$a+x\in A \iff x\in f^1_a(A) \quad \text{ and } \quad a-x\in A \iff x\in f^2_a(A)$$
for all $a\in A$, $x\in X$.
It follows that
$$x\in S(A) \iff \forall_{a\in A} \left(x\in f^1_a(A) \vee x\in f^2_a(A)\right) \iff x\in \bigcap_{a \in A} \left(f_a^1(A) \cup f_a^2(A) \right),$$
which finishes the proof.
%
%
\end{proof}
In the following theorem, we do not have assumption of finiteness of the considered family which was used in the paper \cite{BBP}.
\begin{proposition}
For any family $(A_i)_{i\in I}$ we have $$\bigcap_{i \in I} S(A_i) \subset S(\bigcup_{i \in I} A_i).$$

\end{proposition}
\begin{proof}
Let $x \in \bigcap_{i \in I} S(A_i)$ and $y \in  \bigcup_{i \in I} A_i$. There is $n \in I$ such that $y \in A_n$. Since $x \in S(A_n)$, so $x + y \in A_n$ or $x - y \in A_n$. Of course, $A_n \subset \bigcup_{i\in I} A_i$, therefore $x \in 
S(\bigcup_{i \in I} A_i)$.
\end{proof}
In the next result we omit the assumption of compactness of the considered sets, which was required in the analogue from \cite{BBP}.
\begin{proposition}
For any decreasing sequence $(B_i)_{i\in \N}$ of nonempty sets we have $$\bigcap_{i \in \N} S(B_i) \subset S(\bigcap_{i \in \N} B_i).$$

\end{proposition}
\begin{proof}
Let $x \in \bigcap_{i \in \N} S(B_i)$ and $y \in  \bigcap_{i \in \N} B_i$. For any $n \in \N$, $x+y \in B_n$ or $y-x \in B_n$. Since the sequence $(B_n)$ is decreasing, we have either $x+y \in B_n$ for all $n \in \N$ or $y-x \in B_n$ for all $n \in \N$. Hence $x+y \in \bigcap_{n\in\N}B_n$ or $y-x \in \bigcap_{n\in\N}B_n$. Thus, $x \in S(\bigcap_{i \in \N} B_i).$
\end{proof}

Center of distances was used as a tool to examine achievement sets of some series in $\R$. Thus, it is natural to ask what we can say about the spectre of achievement sets. 

The first result is an easy observation (similar as for center of distances).

\begin{theorem} \label{wyrazy}
Let $\sum_i x_i$ be an absolutely convergent series. We have $x_n \in S(E(x_i))$ for all $n \in \N$.
\end{theorem}
\begin{proof}
Let $n \in \N$ and take $a \in E(x_i)$. There is $B \subset \N$ such that $a=\sum_{i \in B} x_i$. If $n \in B$, then $$a-x_n = \sum_{i \in B\setminus \{n\}} x_i \in E(x_i).$$ If $n \notin B$, then $$a+x_n =\sum_{i \in B\cup \{n\}} x_i\in E(x_i).$$ By the arbitrariness of $a$, $x_n \in S(E(x_i)).$
\end{proof}
\begin{example}
From Example \ref{e1} we know that $$S([0,1]^2)= \left(\{0\} \times [-\frac{1}{2},\frac{1}{2}]\right) \cup \left([-\frac{1}{2},\frac{1}{2}] \times \{0\}\right).$$ From Theorem \ref{wyrazy} we infer that we can obtain the square $[0,1]^2$ as an achievement set only for sequences such that all of their terms are of the form $(0,b_n)$ or $(a_n,0)$ for some sequences $(a_n)$ and $(b_n)$. Actually, we have
$$[0,1]^2 = E(x_n,y_n),$$
where $$(x_n,y_n) := \begin{cases}
(\frac{1}{2^k},0) \;&\text{ if }\; n = 2k \text{  for } k\in \N \\ 
(0,\frac{1}{2^k}) \;&\text{ if }\; n=2k-1 \text{  for } k\in \N.
\end{cases} $$
\end{example}
\begin{example}
Let $B$ be a closed ball in $\R^2.$ Let $a \in \R.$ There is a point $c$ on the boundary of $B$ such that there is $b\in \R$ for which the line $y=ax+b$ is tangent to $B$ at $c$. So, $c+(x,ax) \notin B$ for $x \neq 0$. Hence $(x,ax) \notin S(B)$ for $x \neq 0.$ Similarly, $(0,ax) \notin
S(B)$, so, by the arbitrariness of $a$, $S(B) = \{(0,0)\}$. By Theorem \ref{wyrazy}, $B$ cannot be the achievement set of any series.

\end{example}
\begin{example}
Similarly, as in the previous example we can show that any triangle $T$ cannot be the achievement set of any series in $\R^2$. Considering just the vertices of $T$, we can exclude from $S(T)$ vectors in any direction. So, $S(T) = \{(0,0)\}.$
\end{example}
During the conference "Inspirations in Real Analysis" which was held in 2022 in Będlewo Prof. Andrzej Biś asked the question what is the center of distances of the Sierpiński carpet. The following example partially answers that question (not for the center of distances, but for the spectre).
\begin{example}
Sierpiński carpet $SC$ cannot be obtained as an achievement set of any series in $\R^2$. 
First, we calculate the spectre of the set $SC$. Repeating the reasoning from Example \ref{e1}, we get $S(SC) \subset [-\frac{1}{2},\frac{1}{2}] \times \{0\} \cup \{0\} \times [-\frac{1}{2},\frac{1}{2}].$ Let $C_3$ be a classical Cantor set. Observe that $(c, \frac{1}{2}) \in SC$ if and only if $c \in C_3$. Hence it is possible that $(x,0) \in S(SC)$ only if $x \in C(C_3)$ or $x \in -C(C_3)$. It is known that $C(C_3) = \{\frac{2}{3^n}\colon n \in \N\} \cup \{0\}.$ 

Let $(x_1,y_1) \in SC$ be such that $x_1 \in (\frac{1}{9},\frac{2}{9})$, $y_1 \in (\frac{3}{9},\frac{4}{9})$. Then $(x_1,y_1) + (\frac{2}{9},0) \notin SC$. Also $(x_1,y_1) - (\frac{2}{9},0) \notin SC$, so $(\frac{2}{9},0) \notin S(SC)$. Generally, taking point $(x_n,y_n) \in SC$ such that $x_n \in (\frac{1}{3^n},\frac{2}{3^n})$, $y_1 \in (\frac{3}{3^n},\frac{4}{3^n})$, we get $(x_n,y_n) + (\frac{2}{3^n},0) \notin SC$. Also $(x_n,y_n) - (\frac{2}{3^n},0) \notin SC$, so $(\frac{2}{3^n},0) \notin S(SC)$. So $(x,0) \in S(SC) \Leftrightarrow x =0$. Similarly, we show that $(0,y) \in S(SC) \Leftrightarrow y =0$. Finally, $S(SC) = \{(0,0)\}$, so,
by Theorem \ref{wyrazy}, $SC$ cannot be the achievement set of any series.
  
\end{example}

\begin{theorem}
Let $X = \R^2$. Consider the sequence $(x_n,y_n)$ such that the series $\sum_n (x_n,y_n) = (\sum_n x_n, \sum_n y_n)$ is absolutely convergent. If $(x,y) \in S(E(x_n,y_n))$, then $x\in S(E(x_n))$ and $y \in S(E(y_n))$.
\end{theorem}
\begin{proof}
Let $x,y \in \R$ be such that $(x,y) \in S(E(x_n,y_n))$. Let $z \in E(x_n)$. Then, there is $B \subset \N$ such that $z=\sum_{n\in B} x_n$. But then,
$$\sum_{n \in B} (x_n,y_n) = (z, t) \in E(x_n,y_n), $$ where $t=\sum_{n\in B}y_n$.
Since $(x,y) \in S(E(x_n,y_n))$, we get
$(z+x, t+y) \in E(x_n,y_n)$ or $(z-x, t-y) \in E(x_n,y_n)$. By Propositon \ref{Cart}(i), $E(x_n,y_n) \subset E(x_n) \times E(y_n)$, so 
$x+z \in E(x_n)$ or $z-x \in E(y_n)$. Hence $x \in S(E(x_n))$. Similarly, we show that $y \in S(E(y_n))$.
\end{proof}
The above theorem cannot be reversed, which the next example shows.
\begin{example}
Let $x_n = \frac{1}{4^n}$ and $y_n = \frac{1}{5^n}$ for all $n \in \N$.  Then, by Theorem \ref{wyrazy}, we have $\frac{1}{4} \in S(E(x_n))$ and $\frac{1}{25} \in S(E(y_n))$. 
However, $(\frac{1}{4}, \frac{1}{25}) \notin E(x_n,y_n)$. Indeed, the only way to obtain $\frac{1}{4}$ as a subsum of the series $\sum x_n$ is to use only the first term. Therefore, 
$\{(\frac{1}{4},y) \in \R^2\colon y \in \R\} \cap E(x_n,y_n) = (\frac{1}{4},\frac{1}{5})$. 
Of course, also $(-\frac{1}{4}, -\frac{1}{25}) \notin E(x_n,y_n)$, and since $0 \in E(x_n,y_n)$, we have $(\frac{1}{4}, \frac{1}{25}) \notin S(E(x_n,y_n)).$
\end{example}
Recall that a series $\sum x_n$ is slowly convergent if and only if for any $n \in \N$ we have
$x_n \leq r_n := \sum_{i>n} x_i$. 
In $\R$ we know that $C(E(x_n))$ is an interval if and only if $\sum x_n$ is slowly convergent \cite{BBP}.
There is no analogue for the spectre of an achievement set in $\R^2$.
\begin{example}
Consider the sequence $(\frac{1}{2^n}, \frac{1}{\sqrt{2}^n}).$ It is easy to observe that the sequences $(\frac{1}{2^n})$ and $(\frac{1}{\sqrt{2}^n})$ are slowly convergenet. For $n \in \N$ we have
$$\sum_{i=n+1}^\infty  \frac{1}{\sqrt{2}^i} = \frac{1}{\sqrt{2}^{n+1}}\cdot \frac{\sqrt{2}}{\sqrt{2}-1} =  \frac{2+\sqrt{2}}{\sqrt{2}^{n+1}} \neq \frac{1}{\sqrt{2}^{n}}.$$
By Example \ref{12Cantor}, $E( \frac{1}{2^n}, \frac{1}{\sqrt{2}^n})$ is a Cantor set.

\end{example}

\section{Open questions}
We would like to finish the paper with some open questions.

\begin{question}
    Does there exist absolutely convergent series $\sum_n x_n$, $\sum_n y_n$ such that  $E(x_n,y_n)_0$ is not a set of $P$-sums?
\end{question}







\begin{question}
    How to topologically characterize known types of achievement sets? Can it be done by cuts? Is it true that $E$ is homeomorphic to $M \times M$ if and only if every cut is finite or is an M-Cantorval?
\end{question} 

\begin{question}
    Is it true that any achievement set is either regularly closed or nowhere dense?
\end{question}

\section*{Acknowledgements}
The authors would like to thank Filip Turoboś and Paolo Leonetti for many fruitful discussions. 

This research was funded by University of Lodz from the programme IDUB, Grant number: NR 3/ML/2022.




\end{document}